\documentclass[12pt]{article}

\usepackage{amssymb,amsmath}
\usepackage{xcolor}
\usepackage{amsthm}

\newtheorem{theo}{Theorem}
\newtheorem{coro}[theo]{Corollary}

\newtheorem*{re}{Remark}

\newtheorem{prop}[theo]{Proposition}

\newcommand{\R}{{\bf R}}

\newcommand{\p}{\partial}

\newcommand{\lsm}{\lesssim}
\newcommand{\ep}{\varepsilon}

\newcommand{\al}{\alpha}
\newcommand{\bt}{\beta}

\newcommand{\lm}{\lambda}

\newcommand{\dsp}{\displaystyle}

\newcommand{\nb}{\nabla}
\newcommand{\fa}{\frac}
\newcommand{\sr}{\sqrt}
\newcommand{\lp}[1]{\left(#1 \right)}
\newcommand{\lb}[1]{\left\{#1 \right\}}
\newcommand{\lbt}[1]{\left[#1 \right]}

\begin{document}
 
\begin{center}
{\large\bf 
Blow up of solutions of semilinear wave equations related to nonlinear waves 
in  de Sitter spacetime
}
\end{center}

\vspace{3mm}
\begin{center}

Kimitoshi Tsutaya$^\dagger$ 
        and Yuta Wakasugi$^\ddagger$ \\
\vspace{1cm}

 $^\dagger$Graduate School  of Science and Technology \\
Hirosaki University  \\
Hirosaki 036-8561, Japan\\
\footnotetext{* The research was supported by JSPS KAKENHI Grant Number JP18K03351. }

\vspace{5mm}
        $^\ddagger$
Graduate School of Advanced Science and Engineering \\
Hiroshima University \\
Higashi-Hiroshima, 739-8527, Japan

\end{center}

\begin{abstract}
Consider a nonlinear wave equation  for a massless scalar field  with self-interaction 
in the spatially flat de Sitter spacetime. 
We show that blow-up in a finite time occurs for the equation with arbitrary power nonlinearity as well as upper bounds of the lifespan of blow-up solutions. 
The blow-up condition is the same as in the accelerated expanding Friedmann-Lema\^itre-Robertson-Walker (FLRW) spacetime. 
We also show the same results for the space derivative nonlinear term.  



\end{abstract}

\addtolength{\baselineskip}{2mm}

\section{Introduction}
This paper is subsequent to our recent work \cite{TW1,TW2,TW3,TW4} concerned with the semilinear wave equations in the Friedmann-Lema\^itre-Robertson-Walker (FLRW) spacetimes. 

We consider the following Cauchy problem: 
\begin{equation}
\begin{cases}
\dsp
u_{tt}-a(t)\Delta u+\mu u_t=|u|^p \; \mbox{ or } \; |\nb_x u|^p, &\qquad  t>0, \; x\in \R^n,   \\
  & \\
u(0,x)=\ep u_0(x), \; u_t(0,x)=\ep u_1(x), &\qquad x\in \R^n,  
\end{cases}
\label{eq}
\end{equation}
where $a\in C^1([0,\infty))$ satisfies
\begin{equation}
a(t)> 0, \quad a'(t)\le 0 \quad \mbox{ for all }t\ge 0, \mbox{ and }
\int_0^\infty a(t)^{1/2}dt<\infty, 
\label{asm-a}
\end{equation}
also  $\Delta=\p_1^2+\cdots \p_n^2, \; \nb_xu = (\p_1 u, \cdots,\p_n u), \; \p_j=\p/\p x^j, \; j=1,\cdots,n, \\
 \; (x^1,\cdots, x^n)\in \R^n$,  
$\mu$ is a nonnegative constant, $p>1$, and $\ep>0$ is a small parameter.  

For the special case $a(t) \equiv 1$, there have been many results. 
Especially important among them are Todorova and Yordanov \cite{ToYo01} and Zhang \cite{Zh01} which have proved that the critical exponent is given by
the so-called Fujita exponent 
$p_F(n) = 1+ 2/n$.
The critical exponent means the threshhold condition on $p$ between 
global existence and blow-up of solutions for small initial data.
We refer to the introduction of \cite{IkInWa17} for more details.

For general $a(t)$, D'Abbicco and Lucente \cite{DaLu13} proved that 
if $a(t)$ satisfies the condition
$a(t) \le C (t/\mu)^{-\alpha}$
with some constants
$\alpha < \min\{1,\frac{2}{n}\}$ and $C>0$,
then no global weak solutions exist, provided that
\[
    1 + \frac{\max\{0,\alpha\}}{1-\alpha}
    < p \le
    1 + \frac{2}{n(1-\alpha)}.
\]
In case $a(t) = e^{-2t}$ and $n=\mu =3$,
Galstian \cite{Ga17} showed 
a lower bound of the lifespan of solutions for $2\le p < 5$.
For this case, we can apply the result above by \cite{DaLu13} and 
see that no global weak solutions exist for $1<p <3$ 
by taking arbitrary $\alpha < 2/3$. 
We remark that the result of \cite{DaLu13} cannot be applied to
$p \ge 3$.
To our best knowledge, there are no results of upper bounds of the lifespan of solutions.

For the time-dependent damping case $\mu=\mu(t)$ or the Klein-Gordon equation, 
we refer to, e.g., \cite{BuRe15,EbRe18,Ya09,Ya12}.

The equation in \eqref{eq} generalizes the nonlinear wave equations in the de Sitter spacetime. 
The metric of the de Sitter spacetime with zero space curvature is given by 
\[
g: \; ds^2=-dt^2+e^{2Ht}d\sigma^2, 
\]
where the speed of light is equal to $1$, $d\sigma^2$ is the line element of $n$-dimensional Euclidean space and $H$ is the Hubble constant. 
The scale factor $e^{Ht}$ describes expansion of the spatial metric, 
and also is determined by solving the Einstein equation with the cosmological constant $\Lambda$ and the energy-momentum tensor for the perfect fluid. Let $\rho$ and $p$ be the energy density and pressure, respectively, 
and let us add an equation of state as $p=w\rho$ with a constant $w$. 
If $w=-1$, then we obtain  
the constant density $\rho$ and the scale factor
\[
R(t)=ce^{Ht}, 
\]
with some constant $c$ and $H=\sr{(16\pi G\rho+2\Lambda)/(n(n-1))}$, where $G$ is the gravitational constant. 
See \cite[(3.1)]{CGLY} for the derivation of this scale factor and \cite{GY} for another derivation.

For the spatially flat de Sitter metric,  
the semilinear wave equations $\Box_g u
=\\
|g|^{-1/2}\p_\al(|g|^{1/2}g^{\al\bt}\p_\bt)u= -|u|^p$ or $-|\nb_x u|^p$ with $p>1$ become
\begin{equation}
u_{tt}-e^{-2Ht}\Delta u+nH u_t=|u|^p, \mbox{ or } \; |\nb_x u|^p. 
\label{ore}
\end{equation}
Our aim of this paper is to show that blow-up in a finite time occurs for the generalized equation \eqref{eq} as well as upper bounds of the lifespan of the blow-up solutions. 

Our previous work \cite{TW1,TW2,TW3,TW4} has treated the case $-1<w\le 1$, 
where the background metric is given by 
\begin{equation}
g: \; ds^2=-dt^2+t^{4/(n(1+w))}d\sigma^2, 
\label{FLRW}
\end{equation}
which is the FLRW metric with zero spatial curvature.

We have shown blow-up of solutions in finite time and upper bounds of their lifespan 
for $\Box_g u= -|u|^p$ for case of decelerated expansion $2/n-1<w\le 1, \; n\ge 2$ in \cite{TW1,TW2} 
and for case of accelerated expansion $-1<w \le 2/n-1, \; n\ge 2$ in \cite{TW3}. 
Furthermore, we have studied equation $\Box_g u= -|u_t|^p$ or $-|\nb_x u|^p$ in \cite{TW4}. 
It should be noted here that if $-1<w\le 2/n-1$, then blow-up in finite time can happen to occur for all $p>1$. 

In this paper we show very similar results to the ones above in case $-1<w\le 2/n-1$
for \eqref{eq} generalizing \eqref{ore}. 
In Section 2 we state our main result. 
The first step of our proof is to show the property of finite speed of propagation for classical solutions.  
We then prove the theorem by combining this property with a test function method similar to the one in \cite{IW}. 
We emphasize that unlike \cite{IW}, our test function is of only time variable.  
Finally, we state the conclusion of our work in Section 3.

\section{Main result and proof}
\setcounter{equation}{0}

We now state our main result. 
Let $T_\ep$ be the lifespan of solutions of \eqref{eq}, that is, $T_\ep$ is the supremum of $T$ such that  \eqref{eq} have a solution for $x\in \R^n$ and $0\le t<T$. 

\begin{theo}
 Let $a(t)$ satisfy \eqref{asm-a} and 
 let $n\ge 1, \; \mu\ge 0$ and $p>1$.
Assume that $u_0\in C^2(\R^n)$ and $u_1\in C^1(\R^n)$,  $\mbox{\rm supp }u_0, \mbox{\rm supp }u_1\subset \{|x|\le R\}$ with $R>0$ and 
\[
\int (\mu u_0(x)+u_1(x))dx >0. 
\]
Suppose that the problem \eqref{eq} has  a classical solution $u\in C^2([0,T)\times\R^n)$.  
Then $T<\infty$ and for arbitrary $\ep>0$, the lifespan $T_\ep$ of the solution $u$ to \eqref{eq} is estimated as 
\begin{align*}
 & T_\ep \le C\ep^{-(p-1)} && \mbox{if }\mu>0, \\
 & T_\ep \le C\ep^{-(p-1)/(p+1)} && \mbox{if }\mu=0, 
\end{align*}
where $C>0$ is a constant independent of $\ep$. 
\end{theo}

\begin{re}
{\rm The upper bound of the lifespan for $\mu=0$ is better than that for $\mu>0$ in the theorem.} 
\end{re}

In order to prove Theorem 1,  we first show the following proposition. 
We denote $u'=(u_t,\nb_x u)$, and 
$u''$ represents the vector of the second derivatives of $u$ with respect to time and space variables.

\begin{prop}
Let $F(t,x,u,u',u'')$ be a function of class $C^1$ in $t,x,u,u'$ and $u''$ satisfying 
\begin{equation}
F(t,x,0,0,u'')=0 \quad \mbox{for all }t,x \mbox{ and }u''
\label{asm1}
\end{equation}
and let $u(t,x)$ be a $C^2$-solution of the equation
\begin{equation}
u_{tt}-a(t)\Delta u=F(t,x,u,u',u'')
\label{eqF1}
\end{equation}
with $a\in C^1([0,\infty))$ such that $a(t)>0$ and $a'(t)\le 0$ for all $t\ge 0$ 
in the region 
\[
\Lambda_{T,x_0}=\{(t,x)\in [0,T)\times \R^n: \; |x-x_0|<A(T)-A(t)\}
\]
for some $T>0$ and $x_0\in \R^n$, where 
\begin{equation}
A(t)=\int_0^t \sr{a(s)}ds.  
\label{defA}
\end{equation}
Assume that 
\begin{equation}
u(0,x)=u_t(0,x)=0 \qquad \mbox{for }|x-x_0|<A(T). 
\label{asm2}
\end{equation}
Then $u$ vanishes in $\Lambda_{T,x_0}$. 
\end{prop}
 
\begin{proof} 
We prove the proposition following \cite{Jo1,So}. 
We have only to consider the case $x_0=0$ since the equation is invariant under translation with respect to $x$. 
Let 
\[
\psi(\lm,x)=A^{-1}\lbt{A(T)-\lb{(A(T)-\lm)^2+A(T)^{-2}(2\lm A(T)-\lm^2)|x|^2}^{1/2}}, 
\]
where $A^{-1}$ is the inverse function of $A(t)$, and 
$\lm$ is a parameter between $0$ and $A(T)$. 
We then have 
\[
\psi(0,x)=0, \quad \lim_{\lm\to A(T)}\psi(\lm,x)=A^{-1}(A(T)-|x|). 
\] 
Define the region $R_\lm$ by 
\[
R_\lm=\{(t,x): \; 0\le t\le \psi(\lm,x), \; |x|<A(T)-A(t)\}.   
\]
We have 
\[
\Lambda_{T,0}=\bigcup_{0\le \lm<A(T)}R_\lm
\]
and also 
\begin{align}
&|\nb_x\psi(\lm,x)| \nonumber \\
&=
\left|(A^{-1})'\lp{A(T)-\lb{(A(T)-\lm)^2+A(T)^{-2}(2\lm A(T)-\lm^2)|x|^2}^{1/2}}\right| \nonumber\\
& \qquad \cdot \fa{(2\lm A(T)-\lm^2)|x|}{A(T)^2\lb{(A(T)-\lm)^2+A(T)^{-2}(2\lm A(T)-\lm^2)|x|^2}^{1/2}}
\label{psi1}
\end{align}
for $0\le \lm<A(T)$. 
We here note that 
\begin{align}
&(A^{-1})'\lp{A(T)-\lb{(A(T)-\lm)^2+A(T)^{-2}(2\lm A(T)-\lm^2)|x|^2}^{1/2}}
\nonumber\\
&=\lb{\sr a \circ A^{-1}\lp{A(T)-\lb{(A(T)-\lm)^2+A(T)^{-2}(2\lm A(T)-\lm^2)|x|^2}^{1/2}}}^{-1} \nonumber\\
&=\fa 1{\sr{a(\psi(\lm,x))}}. 
\label{psi2} 
\end{align}
Define the surface $S_\lm$ by  
\[
S_\lm=\{(t,x): t=\psi(\lm,x), \; |x|<A(T)\}. 
\]
The outward unit normal at $(\psi(\lm,x),x)\in S_\lm$ is 
\[
\fa 1{\sr{1+|\nb_x\psi(\lm,x)|^2}}(1,-\nb_x\psi(\lm,x)). 
\]
Then 
\begin{align*}
\int_{R_\lm}2u_tFdtdx
&=\int_{R_\lm}2u_t(u_{tt}-a(t)\Delta u)dtdx \\ 
&=\int_{R_\lm}\bigl\{\p_t(u_t^2+a(t)|\nb_xu|^2)-2\nb_x\cdot(a(t) u_t\nb_x u)\\
& \qquad\qquad  -a'(t)|\nb_x u|^2\bigr\}dtdx \\
&\ge  \int_{R_\lm}\lb{\p_t(u_t^2+a(t)|\nb_xu|^2)-2\nb_x\cdot(a(t) u_t\nb_x u)}dtdx  
\end{align*}
by \eqref{asm-a}. 
Let $\lm_0$ satisfy $0<\lm_0<A(T)$. 
Note by \eqref{psi1} and \eqref{psi2} that 
\begin{align*}
\sr{a(t)}|\nb_x\psi|
&=\fa{(2\lm A(T)-\lm^2)|x|}{A(T)^2\lb{(A(T)-\lm)^2+A(T)^{-2}(2\lm A(T)-\lm^2)|x|^2}^{1/2}} \\
&<A(T)^{-1}\sr{2\lm_0A(T)-\lm_0^2}\equiv \theta(\lm_0)<1
\end{align*} 
on $S_\lm$ for $0\le \lm\le \lm_0<A(T)$. 
Using the divergence theorem, we obtain 
\begin{align}
\int_{R_\lm}2u_tFdtdx
&\ge  \int_{S_\lm}\lb{u_t^2+a(t)|\nb_xu|^2+2\nb_x\psi\cdot(a(t) u_t\nb_x u)} \nonumber \\
& \hspace{6cm}
  \cdot \fa 1{\sr{1+|\nb_x\psi|^2}}d\sigma \nonumber \\
&\ge \int_{S_\lm}\lb{u_t^2+a(t)|\nb_xu|^2-\sr{a(t)}|\nb_x\psi|(u_t^2+a(t)|\nb_x u|^2)} \nonumber \\
& \hspace{6cm} 
 \cdot \fa 1{\sr{1+|\nb_x\psi|^2}}d\sigma \nonumber \\
&\ge (1-\theta(\lm_0))\int_{S_\lm} \fa{u_t^2+a(t)|\nb_xu|^2}{\sr{1+|\nb_x\psi|^2}}d\sigma
\label{Ap1}
\end{align}
for $0\le \lm\le \lm_0<A(T)$.

On the other hand, by assumption \eqref{asm1}, we have 
\[
|u_t F(u,u',u'')|\le C(u^2+|u'|^2)
\]
in $\Lambda_{T,0}$. 
We note that 
\[
\int_0^{\psi(\lm,x)}u(t,x)^2dt\le \fa 12\psi(\lm,x)^2\int_0^{\psi(\lm,x)}u_t(t,x)^2dt
\lsm T^2\int_0^{\psi(\lm,x)}u_t(t,x)^2dt
\]
and also by \eqref{psi2} that 
\begin{align*}
\psi_\lm(\lm,x)
&=(A^{-1})'\lp{A(T)-\lb{(A(T)-\lm)^2+A(T)^{-2}(2\lm A(T)-\lm^2)|x|^2}^{1/2}}  \nonumber\\
&\qquad \cdot
\fa{(A(T)-\lm)(1-A(T)^{-2}|x|^2)}
{\lb{(A(T)-\lm)^2+A(T)^{-2}(2\lm A(T)-\lm^2)|x|^2}^{1/2}} \\
&= \fa 1{\sr{a(\psi(\lm,x))}}\cdot\fa{(A(T)-\lm)(1-A(T)^{-2}|x|^2)}
{\lb{(A(T)-\lm)^2+A(T)^{-2}(2\lm A(T)-\lm^2)|x|^2}^{1/2}},  \\
|\psi_\lm(\lm,x)|
&\le \fa 1{\sr{a(T)}}. 
\end{align*}
We then have 
\begin{align}
\int_{R_\lm}2u_tFdtdx
&\le C(1+T^2)\int_{R_\lm} |u'|^2dtdx \nonumber\\
&\le C_T \int_{R_\lm} (u_t^2+a(t)|\nb_xu|^2)dtdx \nonumber\\
&=C_T\int_0^\lm \int_{S_\mu}(u_t^2+a(t)|\nb_xu|^2)\fa{\psi_\mu}{\sr{1+|\nb_x\psi|^2}}d\sigma d\mu \nonumber\\
&\le C_T\int_0^\lm \int_{S_\mu}\fa{u_t^2+a(t)|\nb_xu|^2}{\sr{1+|\nb_x\psi|^2}}d\sigma d\mu
\label{Ap2}
\end{align} 
for $\lm \le \lm_0$. Set 
\[
I(\lm)=\int_{S_\lm} \fa{u_t^2+a(t)|\nb_xu|^2}{\sr{1+|\nb_x\psi|^2}}d\sigma. 
\]
Using Gronwall's inequality for \eqref{Ap1} and \eqref{Ap2}, we see that 
$I(\lm)=0$ for $0\le \lm\le \lm_0<A(T)$.  Therefore, since $\lm$ and $\lm_0$ are arbitrary, we see that $u'=0$  in $\Lambda_{T,0}$, and thus by \eqref{asm2} also that $u= 0$ in $\Lambda_{T,0}$. 
This completes the proof of the proposition. 
\end{proof}

\vspace{5mm}
From the proposition above, we easily see that the following corollary holds: 

\begin{coro}
Let $F$ in Proposition 2 satisfy \eqref{asm1}. 
If $u$ is a $C^2$-solution of \eqref{eqF1} and if 
$u(0,x)=u_t(0,x)=0$ for $|x|>R$ with $R>0$, then 
$u(t,x)=0$ for $|x|>R+A(t)$. 
\label{cor}
\end{coro}

\begin{proof}[Proof of Theorem 1]
By Corollary \ref{cor}, we see that $C^2$-solution $u$ of \eqref{eq} has the property of finite speed of propagation, and satisfies
\begin{equation}
\mbox{supp }u(t,\cdot)\subset \{|x|\le A(t)+R\}, 
\label{suppu0}
\end{equation}
where $A(t)$ is given by \eqref{defA}. 

We next introduce a test function $\eta\in C_0^\infty([0,\infty)$ satisfying 
\begin{align*}
&0\le \eta(t)\le 1, \qquad \eta(t)=\begin{cases}1 &(0\le t\le \fa 12),  \\
 0 & (t\ge 1), \end{cases} \\
\intertext{and let} 
&\eta^\ast(t)=\chi_{\mbox{\footnotesize supp}\eta'}\eta(t)
=\begin{cases} \eta(t) & (1/2< t< 1),  \\
   0 & \mbox{otherwise},  
  \end{cases} 
\end{align*}
where $\eta'=d\eta/dt$. 
Set 
\[
\psi_\tau(t)=\eta\lp{\fa{t}\tau}^{2p'}, \quad 
\psi_\tau^\ast(t)=\eta^\ast\lp{\fa{t}\tau}^{2p'}
\]
for $0<\tau <T_\ep$. 
Then we see that 
\begin{equation}
|\psi_\tau'|\lsm \tau^{-1}{(\psi_\tau^\ast)}^{1/p}, \quad \mbox{and }\quad 
|\psi_\tau''|\lsm \tau^{-2}{(\psi_\tau^\ast)}^{1/p}.  
\label{est-psi}
\end{equation}
Mutiplying equation in \eqref{eq} by the test funtion $\psi_\tau(t)$, and 
integrating over $\R^n$, we have 
\begin{align}
 &\fa{d^2}{dt^2}\int u\psi_\tau dx-2\fa d{dt}\int  u\psi_\tau' dx
+\int  u\psi_\tau'' dx +\mu\fa{d}{dt}\int u\psi_\tau dx  - \mu\int u\psi_\tau' dx \nonumber \\
=&\int  |u|^p\psi_\tau dx \; \mbox{ or }\; \int  |\nb_x u|^p\psi_\tau dx. 
\label{e3}
\end{align}
For the derivative nonlinearity, we have by Poincar\'e's inequality, 
\[
\int  |\nb_x u|^p\psi_\tau dx \ge \fa{\psi_\tau(t)}{(A(t)+R)^p}\int |u|^pdx
\ge C\int  |u|^p\psi_\tau dx. 
\]
Hence, it suffices to show the theorem for the nonlinearity $|u|^p$.  
Integrating the equality \eqref{e3} in $t$ over $[0,\infty)$, we have 
\begin{align*}
&\int_0^\infty \int u\psi_\tau'' dxdt-\mu\int_0^\infty \int u\psi_\tau' dxdt \\
= & \quad \ep\int\lbt{(\mu\psi_\tau(0)-\psi_\tau'(0))u_0(x)+\psi_\tau(0)u_1(x)}dx 
+\int_0^\infty \int \psi_\tau(t)|u|^p(t,x)dxdt\\
= & \quad \ep\int(\mu u_0(x)+u_1(x))dx +\int_0^\infty \int \psi_\tau(t)|u|^p(t,x)dxdt. 
\end{align*}
Let  
\begin{align*}
I_\tau&\equiv \int_0^\infty \int \psi_\tau(t)|u|^p(t,x) dxdt, \\
J&\equiv \ep\int (\mu u_0(x)+u_1(x))dx>0
\end{align*} 
by assumption. 
Then we have 
\begin{align*}
I_\tau+J &=\int_0^\infty \int u\psi_\tau'' dxdt
-\mu\int_0^\infty \int u\psi_\tau' dxdt \\
&\equiv K_1+K_2. 
\end{align*}
By H\"older's inequality and \eqref{est-psi}, we have 
\begin{align*}
|K_1|&\lsm \tau^{-2}\lp{\int_0^\infty \int |u|^p\psi_\tau^\ast dxdt}^{1/p}
\lp{\int_{\tau/2}^\tau \int_{|x|\le A(t)+R}  dxdt}^{1/{p'}} \\
 &\lsm \tau^{-2+1/{p'}}I_\tau^{1/p}, \\
\intertext{similarly,}
|K_2| &\lsm \mu \tau^{-1+1/{p'}}I_\tau^{1/p}. 
\end{align*}
Thus we have 
\begin{equation}
I_\tau+J \lsm (\tau^{-2+1/{p'}}+\mu\tau^{-1+1/{p'}})I_\tau^{1/p}. 
\label{CE}
\end{equation}
We set 
\[
E(\tau)=\tau^{-2+1/{p'}}+\mu\tau^{-1+1/{p'}}. 
\] 
By Young's inequality, we have 
\[
E(\tau)I_\tau^{1/p}\le \fa 12I_\tau+CE(\tau)^{p'}.
\]
It follows from \eqref{CE} that 
\[
J\le CE(\tau)^{p'},  
\]
hence, 
\[
\ep\le CE(\tau)^{p'}
\]
holds for $0<\tau<T_\ep$. 
Thus we have
\[
\ep\le
\begin{cases} 
 C\tau^{-1/(p-1)} &\mbox{if }\mu>0, \\
 C\tau^{-(p+1)/(p-1)} &\mbox{if }\mu=0, 
\end{cases}
\]
which imply that 
\[
\tau\le 
\begin{cases}
C\ep^{-(p-1)}&\mbox{if }\mu>0, \\
C\ep^{-(p-1)/(p+1)}&\mbox{if }\mu=0. 
\end{cases}
\]
Since $\tau$ is arbitrary in $(0,T_\ep)$, we obtain the desired estimates of the lifespan. 
This completes the proof of the theorem. 
\end{proof}

Finally, we remark that it is possible to obtain similar blow-up results to the theorem for weak solutions by using only test function methods in \cite{IkSo19,IW,W}. 
However, an upper bound for the $p$-blow-up range is required. 
We thus refrain from going into the details. 
 
\section{Conclusion}
\setcounter{equation}{0}

We consider the original equation \eqref{ore} in the de Sitter spacetime. 
We see from Theorem 1 that blow-up of solutions in finite time occurs for all $p>1$. 
This blow-up condition is exactly the same as that for $\Box_g u= -|u|^p$ or $-|\nb_x u|^p$ in the accelerated expanding FLRW spacetime with metric \eqref{FLRW} ($-1<w \le 2/n-1$) as stated in Section 1.   
Moreover, we have shown in \cite{TW3,TW4} the following upper bound of the lifespan 
for $\Box_g u= -|u|^p$ or $-|\nb_x u|^p$: 
\[
T_\ep\le C\ep^{-(p-1)/2} \qquad \mbox{if }p>1\mbox{ and }-1<w < 2/n-1. 
\]
This estimate is similar to the one in Theorem 1. 
Therefore, we conjecture that in the accelerated expanding universe the nonlinear wave equation with the term $|u|^p$ or $|\nb_x u|^p$ admits blow-up solutions for all $p>1$.

\vspace{3mm}
\noindent
{\large\bf Acknowledgements}

The authors would like to thank the referee for his/her valuable comments and suggestions on the first version of this paper. 
They would also like to thank Professor Katayama whose question has led to an improvement of the estimates of the lifespan in Theorem 1.


\end{document}